\documentclass[10pt,portrait,reqno,11pt]{amsart}%
\usepackage{amssymb,amsthm,amsmath, amsfonts, float}
\usepackage{graphicx}
\usepackage{multirow}
\usepackage{subfig}
\usepackage{amsmath}
\usepackage{amsfonts}
\usepackage{lscape}
\usepackage{amssymb}%
\providecommand{\U}[1]{\protect\rule{.1in}{.1in}}
\newtheorem{theorem}{Theorem}
\theoremstyle{plain}

\numberwithin{equation}{section}
\begin{document}
\title[Canal Hypersurfaces Generated by Non-Null Curves with Parallel Frame in
$E_{1}^{4}$]{Canal Hypersurfaces Generated by Non-Null Curves with Parallel Frame in
Minkowski Space-Time}
\author[M. Alt\i n, A. Kazan and D.W. Yoon]{\bfseries Mustafa Alt\i n$^{1\ast}$, Ahmet Kazan$^{2}$ and Dae Won Yoon$^{3}$}
\address{ \\
$^{1}$ Department of Mathematics, Faculty of Arts and Sciences, Bing\"{o}l
University, Bing\"{o}l, Turkey \\
 \\
$^{2}$ Department of Computer Technologies, Do\u{g}an\c{s}ehir Vahap
K\"{u}\c{c}\"{u}k Vocational School, Malatya Turgut \"{O}zal University,
Malatya, Turkey\\
 \\
$^{3}$Department of Mathematics Education and RINS, Gyeongsang National
University, Jinju 52828, Republic of Korea \\
 \\
$^{\ast}$Corresponding author: maltin@binol.edu.tr}

\begin{abstract}
In the present paper, firstly we obtain the canal hypersurfaces that are
formed as the envelope of a family of pseudo hyperspheres or pseudo hyperbolic
hyperspheres whose centers lie on a spacelike curve with parallel timelike
normal vector field $B_{2}$ in Minkowski space-time and we give some geometric
characterizations for them by obtaining the Gaussian curvature, mean curvature
and principal curvatures of these canal hypersurfaces. Also, we give the
general expression of parametrizations of the canal hypersurfaces generated by
non-null curves with parallel frame in $E_{1}^{4}$ and we obtain some
important geometric invariants and characterizations for them.

\end{abstract}
\maketitle


\section{\textbf{INTRODUCTION}}

Because the extrinsic differential geometry of submanifolds in Minkowski
space-time $E_{1}^{4}$ is of special interest in Relativity Theory, many
papers about curves and (hyper)surfaces have been done in $E_{1}^{4}$. When
the differential geometers studied on curves and (hyper)surfaces in $E_{1}%
^{4}$, they usually have used the Frenet--Serret frame. Since this frame is
undefined when the second derivative of the curve vanishes, an alternative
frame (called the parallel frame) which is well defined even if the curve has
vanishing second derivative has been inroduced by R.L. Bishop in 1975
\cite{Bishop}. After defining this frame, many studies about curves and
(hyper)surfaces according to parallel frame in different spaces have been done
by geometers (see \cite{Ates}, \cite{Buyuk}, \cite{Melek}, \cite{Fatma},
\cite{GenBishop}, \cite{HBK}, \cite{Nomoto}, and etc.) and this motivated us
to construct the canal hypersurfaces generated by non-null curves with
parallel frame in Minkowski space-time.

Also, Monge has first investigated the canal surfaces which are class of
surfaces formed by sweeping a sphere and so, one can see a canal surface as
the envelope of a moving sphere with varying radius, defined by the trajectory
$C(u)$ of its centers and a radius function $r(u)$ and canal surface is
parametrized through Frenet frame of the spine curve $C(u)$. If the radius
function $r(u)$ is constant, then the canal surface is called a tubular (or
pipe) surface. The canal and tubular surfaces are very useful tools for
geometers and engineers. They can be used in many application areas such as
the solid and the surface modeling for CAD/CAM, construction of blending
surfaces, shape re-construction and they are useful to represent various
objects such as pipe, hose, rope or intestine of a body (\cite{Karacany},
\cite{Ucum}). \ For different studies about canal or tubular (hyper)surfaces
in different spaces, one can see \cite{Mahmut}, \cite{Aslan}, \cite{Yusuf},
\cite{Fu}, \cite{Garcia}, \cite{Hartman}, \cite{Izumiya}, \cite{Karacan},
\cite{Karacan2}, \cite{Karacan3}, \cite{Karacany}, \cite{Altin}, \cite{Sezai},
\cite{Kim}, \cite{Kisi}, \cite{Krivos}, \cite{Maekawa}, \cite{Peter},
\cite{Ro}, \cite{Ucum}, \cite{Xu}, \cite{Kucuk}, and etc.

In this study, we obtain the canal hypersurfaces that are formed as the
envelope of a family of pseudo hyperspheres or pseudo hyperbolic hyperspheres
whose centers lie on a spacelike curve with parallel timelike normal vector
field $B_{2}$ in $E_{1}^{4}$ and with the aid of this parametrization, we
obtain the Gaussian curvature, mean curvature and principal curvatures of
these canal hypersurfaces. Also, we give the necessary and sufficiant
conditions for the flatness and minimality of these canal hypersurfaces.
Moreover we give the general expression of the canal hypersurfaces generated
by non-null curves with parallel frame in $E_{1}^{4}$ and we obtain some
important geometric invariants and characterizations for these canal hypersurfaces.

\section{\textbf{PRELIMINARIES}}

Let $\overrightarrow{x}=(x_{1},x_{2},x_{3},x_{4})$, $\overrightarrow{y}%
=(y_{1},y_{2},y_{3},y_{4})$ and $\overrightarrow{z}=(z_{1},z_{2},z_{3},z_{4})$
be three vectors in Minkowski space-time. The inner product of two vectors is
defined by%
\begin{equation}
\left\langle \overrightarrow{x},\overrightarrow{y}\right\rangle =-x_{1}%
y_{1}+x_{2}y_{2}+x_{3}y_{3}+x_{4}y_{4} \label{yy1}%
\end{equation}
and the vector product of three vectors is defined by
\begin{equation}
\overrightarrow{x}\times\overrightarrow{y}\times\overrightarrow{z}=\det\left[
\begin{array}
[c]{cccc}%
-e_{1} & e_{2} & e_{3} & e_{4}\\
x_{1} & x_{2} & x_{3} & x_{4}\\
y_{1} & y_{2} & y_{3} & y_{4}\\
z_{1} & z_{2} & z_{3} & z_{4}%
\end{array}
\right]  , \label{yy2}%
\end{equation}
where $e_{i},(i=1,2,3,4)$ are standard basis vectors. The norm of the vector
$\overrightarrow{x}$ is $\left\Vert \overrightarrow{x}\right\Vert
=\sqrt{\left\vert \left\langle \overrightarrow{x},\overrightarrow{x}%
\right\rangle \right\vert }$.

Also, a vector $\overrightarrow{x}\in E_{1}^{4}$ is called spacelike if
$\left\langle \overrightarrow{x},\overrightarrow{x}\right\rangle >0$ (or
$\overrightarrow{x}=0$); timelike if $\left\langle \overrightarrow{x}%
,\overrightarrow{x}\right\rangle <0$ and lightlike (null) if $\left\langle
\overrightarrow{x},\overrightarrow{x}\right\rangle =0,$ $\overrightarrow{x}%
\neq0$. A curve $\gamma(u)$ in $E_{1}^{4}$ is spacelike, timelike or lightlike
(null), if all its velocity vectors $\gamma^{\prime}(u)$ are spacelike,
timelike or lightlike, respectively and a timelike or spacelike (i.e.
non-null) curve has unit speed if $\left\langle \gamma^{\prime},\gamma
^{\prime}\right\rangle =\mp1$ \cite{Kuhnel}, \cite{Oneill}.

If $\{B_{1},B_{2},B_{3},B_{4}\}$ is the parallel frame along the non-null
curve $\gamma(u)$ in $E_{1}^{4}$, then the following equations can be given
according to the causal characters of non-null parallel vector fields
\cite{Melek}:

If $\gamma(u)$ is a unit speed timelike curve, then%
\begin{equation}
\left.
\begin{array}
[c]{l}%
B_{1}^{\prime}(u)=k_{1}(u)B_{2}(u)+k_{2}(u)B_{3}(u)+k_{3}(u)F_{4}(u),\\
B_{2}^{\prime}(u)=k_{1}(u)B_{1}(u),\\
B_{3}^{\prime}(u)=k_{2}(u)B_{1}(u),\\
B_{4}^{\prime}(u)=k_{3}(u)B_{1}(u),
\end{array}
\right\}  \label{Bishop-+++}%
\end{equation}%
\begin{equation}
\kappa=\sqrt{\left(  k_{1}(u)\right)  ^{2}+\left(  k_{2}(u)\right)
^{2}+\left(  k_{3}(u)\right)  ^{2}}; \label{kappa1}%
\end{equation}
if $\gamma(u)$ is a unit speed spacelike curve with parallel timelike normal
vector field $B_{2}$, then%
\begin{equation}
\left.
\begin{array}
[c]{l}%
B_{1}^{\prime}(u)=k_{1}(u)B_{2}(u)+k_{2}(u)B_{3}(u)+k_{3}(u)F_{4}(u),\\
B_{2}^{\prime}(u)=k_{1}(u)B_{1}(u),\\
B_{3}^{\prime}(u)=-k_{2}(u)B_{1}(u),\\
B_{4}^{\prime}(u)=-k_{3}(u)B_{1}(u),
\end{array}
\right\}  \label{Bishop+-++}%
\end{equation}%
\begin{equation}
\kappa=\sqrt{\left(  k_{1}(u)\right)  ^{2}-\left(  k_{2}(u)\right)
^{2}-\left(  k_{3}(u)\right)  ^{2}}; \label{kappa2}%
\end{equation}
if $\gamma(u)$ is a unit speed spacelike curve with parallel timelike normal
vector field $B_{3}$, then%
\begin{equation}
\left.
\begin{array}
[c]{l}%
B_{1}^{\prime}(u)=k_{1}(u)B_{2}(u)+k_{2}(u)B_{3}(u)+k_{3}(u)F_{4}(u),\\
B_{2}^{\prime}(u)=-k_{1}(u)B_{1}(u),\\
B_{3}^{\prime}(u)=k_{2}(u)B_{1}(u),\\
B_{4}^{\prime}(u)=-k_{3}(u)B_{1}(u),
\end{array}
\right\}  \label{Bishop++-+}%
\end{equation}%
\begin{equation}
\kappa=\sqrt{\left(  k_{1}(u)\right)  ^{2}-\left(  k_{2}(u)\right)
^{2}+\left(  k_{3}(u)\right)  ^{2}}; \label{kappa3}%
\end{equation}
if $\gamma(u)$ is a unit speed spacelike curve with parallel timelike normal
vector field $B_{4}$, then%
\begin{equation}
\left.
\begin{array}
[c]{l}%
B_{1}^{\prime}(u)=k_{1}(u)B_{2}(u)+k_{2}(u)B_{3}(u)+k_{3}(u)F_{4}(u),\\
B_{2}^{\prime}(u)=-k_{1}(u)B_{1}(u),\\
B_{3}^{\prime}(u)=-k_{2}(u)B_{1}(u),\\
B_{4}^{\prime}(u)=k_{3}(u)B_{1}(u),
\end{array}
\right\}  \label{Bishop+++-}%
\end{equation}%
\begin{equation}
\kappa=\sqrt{\left(  k_{1}(u)\right)  ^{2}+\left(  k_{2}(u)\right)
^{2}-\left(  k_{3}(u)\right)  ^{2}}, \label{kappa4}%
\end{equation}
where $\kappa$ is the curvature and $k_{1},k_{2},k_{3}$ are the first, second
and third principal curvature functions of the non-null curve $\gamma(u)$
according to the parallel frame.

Moreover, if $p$ is a fixed point in $E_{1}^{4}$ and $r$ is a positive
constant, then the pseudo-Riemannian hypersphere and the pseudo-Riemannian
hyperbolic space are defined by%
\[
S_{1}^{3}(p,r)=\{x\in E_{1}^{4}:\left\langle x-p,x-p\right\rangle =r^{2}\}
\]
and
\[
H_{0}^{3}(p,r)=\{x\in E_{1}^{4}:\left\langle x-p,x-p\right\rangle =-r^{2}\},
\]
respectively. In the present study, we construct the canal hypersurfaces,
generated by non-null curves with parallel frame in Minkowski space-time, as
the envelope of a family of pseudo hyperspheres or pseudo hyperbolic
hyperspheres whose centers lie on a non-null curve with non-null parallel
vector fields.

On the other hand, let
\begin{align*}
\Psi:U\subset E^{3}  &  \longrightarrow E_{1}^{4}\\
(x_{1},x_{2},x_{3})  &  \longrightarrow\Psi(x_{1},x_{2},x_{3})=\left(
\Psi_{1}(x_{1},x_{2},x_{3}),\Psi_{2}(x_{1},x_{2},x_{3}),\Psi_{3}(x_{1}%
,x_{2},x_{3}),\Psi_{4}(x_{1},x_{2},x_{3})\right)
\end{align*}
be a hypersurface in Minkowski space-time. The unit normal vector field, the
matrix forms of the first and second fundamental forms of the hypersurface
$\Psi(x_{1},x_{2},x_{3})$ in $E_{1}^{4}$ are%
\begin{equation}
N=\frac{\Psi_{x_{1}}\times\Psi_{x_{2}}\times\Psi_{x_{3}}}{\left\Vert
\Psi_{x_{1}}\times\Psi_{x_{2}}\times\Psi_{x_{3}}\right\Vert }, \label{normal}%
\end{equation}%
\begin{equation}
\lbrack g_{ij}]=\left[
\begin{array}
[c]{ccc}%
g_{11} & g_{12} & g_{13}\\
g_{21} & g_{22} & g_{23}\\
g_{31} & g_{32} & g_{33}%
\end{array}
\right]  \label{gij}%
\end{equation}
and%
\begin{equation}
\lbrack h_{ij}]=\left[
\begin{array}
[c]{ccc}%
h_{11} & h_{12} & h_{13}\\
h_{21} & h_{22} & h_{23}\\
h_{31} & h_{32} & h_{33}%
\end{array}
\right]  , \label{hij}%
\end{equation}
respectively. Here $g_{ij}=\left\langle \Psi_{x_{i}},\Psi_{x_{j}}\right\rangle
,$ $h_{ij}=\left\langle \Psi_{x_{i}x_{j}},N\right\rangle ,$ $\Psi_{x_{i}%
}=\frac{\partial\Psi}{\partial x_{i}}$, $\Psi=\Psi(x_{1},x_{2},x_{3})$ and
$\Psi_{x_{i}x_{j}}=\frac{\partial^{2}\Psi}{\partial x_{i}\partial x_{j}},$
$i,j\in\{1,2,3\}$.

Also, the matrix of shape operator of the hypersurface $\Psi$ is%
\begin{equation}
S=[S_{ij}]=[g_{ij}]^{-1}.[h_{ij}], \label{sekilop}%
\end{equation}
where $[g_{ij}]^{-1}$ is the inverse matrix of $[g_{ij}]$.

With the aid of (\ref{gij})-(\ref{sekilop}), the Gaussian curvature and mean
curvature of a hypersurface in $E_{1}^{4}$ are given for $\varepsilon
=\left\langle N,N\right\rangle $ by%
\begin{equation}
K=\varepsilon\frac{\det[h_{ij}]}{\det[g_{ij}]} \label{Gaussformula}%
\end{equation}
and%
\begin{equation}
H=\varepsilon\frac{tr(S)}{3}, \label{meanformula}%
\end{equation}
respectively. A hypersurface is called flat or minimal, if it has zero
Gaussian or zero mean curvature, respectively. With the aid of these formulas,
the differential geometry of different types of (hyper)surfaces in
4-dimensional spaces has been a popular topic for geometers, recently
(\cite{Altin4}, \cite{Altin2}, \cite{Altin3}, \cite{Aydin}, \cite{Altin},
\cite{Kisi}, and etc.).

\section{\textbf{CANAL HYPERSURFACES GENERATED BY NON-NULL CURVES WITH
PARALLEL FRAME IN }$E_{1}^{4}$}

In this section, firstly we obtain the canal hypersurfaces that are formed as
the envelope of a family of pseudo hyperspheres or pseudo hyperbolic
hyperspheres whose centers lie on a spacelike curve with parallel timelike
normal vector field $B_{2}$ in $E_{1}^{4}$. With the aid of this
parametrization, we obtain the Gaussian curvature, mean curvature and
principal curvatures of these canal hypersurfaces and give some geometric
characterizations for them. After that we give the general expression, which
contains 8 types of parametrizations, of the canal hypersurfaces generated by
non-null curves with parallel frame in $E_{1}^{4}$. We obtain some important
geometric invariants and characterizations for these canal hypersurfaces.

Now, let us prove the following theorem which contains the canal hypersurfaces
that are formed as the envelope of a family of pseudo hyperspheres or pseudo
hyperbolic hyperspheres whose centers lie on a spacelike curve with parallel
timelike normal vector field $B_{2}$ in $E_{1}^{4}$.

\begin{theorem}
Canal hypersurfaces generated by a spacelike curve with parallel timelike
normal vector field $B_{2}$ in $E_{1}^{4}$ can be parametrized by%
\begin{align}
C_{n}(u,v,w)  &  =\gamma(u)+(-1)^{n}r(u)r^{\prime}(u)B_{1}(u)\label{C12}\\
&  \pm r(u)\sqrt{(-1)^{n}+r^{\prime}(u)^{2}}\left(  \cosh v\cosh
wB_{2}(u)+\sinh wB_{3}(u)+\sinh v\cosh wB_{4}(u)\right)  ,\nonumber
\end{align}
where $n\in\{1,2\},$ $r^{\prime}(u)^{2}>-(-1)^{n}$. Also, $C_{1}$ and $C_{2}$
state the canal hypersurfaces that are formed as the envelope of a family of
pseudo hyperspheres and pseudo hyperbolic hyperspheres, respectively, whose
centers lie on a spacelike curve $\gamma(s)$ with parallel timelike normal
vector field $B_{2}$.
\end{theorem}

\begin{proof}
Let $\gamma:I\subseteq%
\mathbb{R}
\rightarrow E_{1}^{4}$ be a spacelike curve with parallel timelike normal
vector field $B_{2}$ parametrized by arc-length with non-zero curvature. Then,
the parametrization of the envelope of pseudo hyperspheres or pseudo
hyperbolic hyperspheres defining the canal hypersurfaces $C_{n}$ in $E_{1}%
^{4}$ can be given by%
\begin{equation}
C_{n}(u,v,w)-\gamma(u)=a_{1}(u,v,w)B_{1}(u)+a_{2}(u,v,w)B_{2}(u)+a_{3}%
(u,v,w)B_{3}(u)+a_{4}(u,v,w)B_{4}(u), \label{Cn-gama}%
\end{equation}
where $a_{i}(u,v,w)$ are differentiable functions of $u,v,w$ on the interval
$I$. Furthermore, since $C_{1}(u,v,w)$ and $C_{2}(u,v,w)$ lies on the pseudo
hyperspheres and pseudo hyperbolic hyperspheres, respectively, we have%
\begin{equation}
g(C_{n}(u,v,w)-\gamma(u),C_{n}(u,v,w)-\gamma(u))=-(-1)^{n}r^{2}(u) \label{2}%
\end{equation}
and so, from (\ref{Cn-gama}), we get%
\begin{equation}
a_{1}^{2}-a_{2}^{2}+a_{3}^{2}+a_{4}^{2}=-(-1)^{n}r^{2} \label{3}%
\end{equation}
and%
\begin{equation}
a_{1}a_{1_{u}}-a_{2}a_{2_{u}}+a_{3}a_{3_{u}}+a_{4}a_{4_{u}}=-(-1)^{n}%
rr^{\prime}, \label{4}%
\end{equation}
where $r(u)$ is the radius function; $r=r(u),$ $r^{\prime}=\frac{dr(u)}{du},$
$a_{i}=a_{i}(u,v,w),$ $a_{i_{u}}=\frac{\partial a_{i}(u,v,w)}{\partial u}$.

If we differentiate the equation (\ref{Cn-gama}) with respect to $u$ and use
(\ref{Bishop+-++}), we have%
\begin{equation}
(C_{n})_{u}=\left(  1+a_{1_{u}}+a_{2}k_{1}-a_{3}k_{2}-a_{4}k_{3}\right)
B_{1}+\left(  a_{1}k_{1}+a_{2_{u}}\right)  B_{2}+\left(  a_{1}k_{2}+a_{3_{u}%
}\right)  B_{3}+\left(  a_{1}k_{3}+a_{4_{u}}\right)  B_{4} \label{Cnu}%
\end{equation}
where $(C_{n})_{u}=\frac{\partial C_{n}(u,v,w)}{\partial u}$. Also,
$C_{n}(u,v,w)-\gamma(u)$ is a normal vector to the canal hypersurfaces, which
implies that%
\begin{equation}
g\left(  C_{n}(u,v,w)-\gamma(u),(C_{n})_{u}(u,v,w)\right)  =0 \label{6}%
\end{equation}
and so, from (\ref{Cn-gama}), (\ref{Cnu}) and (\ref{6}) we have%
\begin{equation}
{\small a}_{1}\left(  1+a_{1_{u}}+a_{2}k_{1}-a_{3}k_{2}-a_{4}k_{3}\right)
-{\small a}_{2}\left(  a_{1}k_{1}+a_{2_{u}}\right)  {\small +a}_{3}\left(
a_{1}k_{2}+a_{3_{u}}\right)  {\small +a}_{4}\left(  a_{1}k_{3}+a_{4_{u}%
}\right)  {\small =0.} \label{6y}%
\end{equation}
Using (\ref{4}) in (\ref{6y}), we get
\begin{equation}
a_{1}=(-1)^{n}rr^{\prime}. \label{7y}%
\end{equation}
Hence, using (\ref{7y}) in (\ref{3}), we reach that%
\begin{equation}
-a_{2}^{2}+a_{3}^{2}+a_{4}^{2}=-r^{2}((-1)^{n}+r^{\prime2}) \label{8}%
\end{equation}
and so, $a_{i},$ $i=2,3,4,$ can be chosen as (\ref{C12}). This completes the proof.
\end{proof}

From now on, we will give our results by taking the "$\pm$" stated in
(\ref{C12}) as "$+$" and one can obtain the results by taking the "$\pm$"
stated in (\ref{C12}) as "$-$". Also, we state $r^{\prime\prime}=\frac
{d^{2}r(u)}{du^{2}},$ $r^{\prime\prime\prime}=\frac{d^{3}r(u)}{du^{3}},$ and
so on. On the other hand, throughout this study, we will assume that $r$ is
not constant.

From (\ref{Bishop+-++}) and (\ref{C12}), we have%
\begin{align}
(C_{n})_{u}  &  =\left(  {r\left(  \sqrt{(-1)^{n}+r^{\prime2}}\mathcal{A}%
+(-1)^{n}r^{\prime\prime}\right)  +(-1)^{n}r^{\prime2}+1}\right)
B_{1}\label{C12u}\\
&  +{r^{\prime}\left(  (-1)^{n}rk_{1}+\frac{\left(  (-1)^{n}+r^{\prime
2}+rr^{\prime\prime}\right)  \cosh v\cosh w}{\sqrt{(-1)^{n}+r^{\prime2}}%
}\right)  }B_{2}\nonumber\\
&  +{r^{\prime}\left(  (-1)^{n}rk_{2}+\frac{\left(  (-1)^{n}+r^{\prime
2}+rr^{\prime\prime}\right)  \sinh w}{\sqrt{(-1)^{n}+r^{\prime2}}}\right)
}B_{3}\nonumber\\
&  +{r^{\prime}\left(  (-1)^{n}rk_{3}+\frac{\left(  (-1)^{n}+r^{\prime
2}+rr^{\prime\prime}\right)  \sinh v\cosh w}{\sqrt{(-1)^{n}+r^{\prime2}}%
}\right)  }B_{4},\nonumber
\end{align}%
\begin{equation}
(C_{n})_{v}=r\sqrt{(-1)^{n}+r^{\prime2}}\left(  {\sinh v\cosh w}B_{2}+{\cosh
v\cosh w}B_{4}\right)  \label{C12v}%
\end{equation}
and%
\begin{equation}
(C_{n})_{w}=r\sqrt{(-1)^{n}+r^{\prime2}}\left(  {\cosh v\sinh w}B_{2}+{\cosh
w}B_{3}+{\sinh v\sinh w}B_{4}\right)  , \label{C12wy}%
\end{equation}
where%
\[
\mathcal{A}=\mathcal{A}(u,v,w)\mathcal{=}k_{1}\cosh v\cosh w-k_{2}\sinh
w-k_{3}\sinh v\cosh w.
\]
Using (\ref{C12u})-(\ref{C12wy}) in (\ref{normal}), we get the unit normal
vector fields of the canal hypersurfaces $C_{n}$ parametrized by (\ref{C12})
as%
\begin{equation}
N_{n}=-\left(  r^{\prime}B_{1}+(-1)^{n}\sqrt{(-1)^{n}+r^{\prime}{}^{2}}\left(
\cosh v\cosh wB_{2}+\sinh wB_{3}+\sinh v\cosh wB_{4}\right)  \right)  .
\label{NC12}%
\end{equation}
From (\ref{gij}) and (\ref{C12u})-(\ref{C12wy}), we obtain the coefficients of
the first fundamental forms of the canal hypersurfaces $C_{1}$ and $C_{2}$,
respectively, as%
\begin{equation}
\left.
\begin{array}
[c]{l}%
\left(  g_{11}\right)  _{n}=\left(  r\left(  \sqrt{(-1)^{n}+r^{\prime2}%
}\mathcal{A}+(-1)^{n}r^{\prime\prime}\right)  +(-1)^{n}r^{\prime2}+1\right)
^{2}\\
\text{ \ \ \ \ \ \ \ \ \ }-r^{\prime2}\left(
\begin{array}
[c]{l}%
\left(  rk_{1}+(-1)^{n}\frac{\left(  (-1)^{n}+r^{\prime2}+rr^{\prime\prime
}\right)  \cosh v\cosh w}{\sqrt{(-1)^{n}+r^{\prime2}}}\right)  ^{2}-\left(
rk_{2}+(-1)^{n}\frac{\left(  (-1)^{n}+r^{\prime2}+rr^{\prime\prime}\right)
\sinh w}{\sqrt{(-1)^{n}+r^{\prime2}}}\right)  ^{2}\\
-\left(  rk_{3}+(-1)^{n}\frac{\left(  (-1)^{n}+r^{\prime2}+rr^{\prime\prime
}\right)  \sinh v\cosh w}{\sqrt{(-1)^{n}+r^{\prime2}}}\right)  ^{2}%
\end{array}
\right)  ,\\
\\
\left(  g_{12}\right)  _{n}=\left(  g_{21}\right)  _{n}=(-1)^{n}r^{2}%
r^{\prime}\sqrt{(-1)^{n}+r^{\prime2}}(k_{3}\cosh v-k_{1}\sinh v)\cosh w,\\
\\
\left(  g_{13}\right)  _{n}=\left(  g_{31}\right)  _{n}=(-1)^{n}r^{2}%
r^{\prime}\sqrt{(-1)^{n}+r^{\prime2}}(k_{2}\cosh w+(k_{3}\sinh v-k_{1}\cosh
v)\sinh w),\\
\\
\left(  g_{22}\right)  _{n}=r^{2}\left(  (-1)^{n}+r^{\prime2}\right)
\cosh^{2}w,\text{ \ }\left(  g_{23}\right)  _{n}=\left(  g_{32}\right)
_{n}=0,\text{ \ }\left(  g_{33}\right)  _{n}=r^{2}\left(  (-1)^{n}+r^{\prime
2}\right)
\end{array}
\right\}  \label{gijC12}%
\end{equation}
and so,%
\begin{equation}
\det\left[  g_{ij}\right]  _{n}=(-1)^{n}r^{4}\left(  (-1)^{n}+r^{\prime
2}\right)  \left(  r\left(  (-1)^{n}\sqrt{(-1)^{n}+r^{\prime2}}\mathcal{A}%
+r^{\prime\prime}\right)  +(-1)^{n}+r^{\prime2}\right)  ^{2}\cosh^{2}w.
\label{detgijC12}%
\end{equation}

Also, from (\ref{Bishop+-++}) and (\ref{C12u})-(\ref{C12wy}), we get the
second derivatives of $C_{n}$ parametrized by (\ref{C12}) as%
\begin{align}
(C_{n})_{uu}  &  =\left(
\begin{array}
[c]{l}%
(-1)^{n}r{r^{\prime}\left(  k_{1}^{2}-k_{2}^{2}-k_{3}^{2}\right)  +}%
r\sqrt{(-1)^{n}+r^{\prime2}}\mathcal{A}_{u}\\
{+2r^{\prime}\sqrt{(-1)^{n}+r^{\prime2}}\mathcal{A}+\frac{2rr^{\prime
}r^{\prime\prime}\mathcal{A}}{\sqrt{(-1)^{n}+r^{\prime2}}}}+(-1)^{n}\left(
r{r}^{\prime\prime\prime}{+3r^{\prime}r^{\prime\prime}}\right)
\end{array}
\right)  B_{1}\label{C12uu}\\
&  +{\frac{\left(
\begin{array}
[c]{l}%
k_{1}^{2}r\left(  (-1)^{n}+r^{\prime2}\right)  ^{2}\cosh v\cosh w+\left(
3r^{\prime4}+(-1)^{n}4r^{\prime2}+1\right)  r^{\prime\prime}\cosh v\cosh w\\
+(-1)^{n}r\left(  k_{1}^{\prime}r^{\prime}\left(  (-1)^{n}+r^{\prime2}\right)
^{3/2}+\left(  r^{\prime\prime2}+(-1)^{n}r^{\prime}r^{\prime\prime\prime
}\left(  (-1)^{n}+r^{\prime2}\right)  \right)  \cosh v\cosh w\right) \\
+(-1)^{n}\left(  (-1)^{n}+r^{\prime2}\right)  k_{1}\left(
\begin{array}
[c]{l}%
-(-1)^{n}r\left(  (-1)^{n}+r^{\prime2}\right)  (k_{2}\sinh w+k_{3}\sinh v\cosh
w)\\
+\sqrt{(-1)^{n}+r^{\prime2}}\left(  2rr^{\prime\prime}+2r^{\prime2}%
+(-1)^{n}\right)
\end{array}
\right)
\end{array}
\right)  }{\left(  (-1)^{n}+r^{\prime2}\right)  ^{3/2}}}B_{2}\nonumber\\
&  +\left(
\begin{array}
[c]{l}%
{-}r{\sqrt{(-1)^{n}+r^{\prime2}}k_{2}^{2}\sinh w}\\
{+k_{2}\left(  r\left(  \sqrt{(-1)^{n}+r^{\prime2}}(k_{1}\cosh v-k_{3}\sinh
v)\cosh w+(-1)^{n}2r^{\prime\prime}\right)  +(-1)^{n}2r^{\prime2}+1\right)
}\\
{+\frac{\left(  3r^{\prime4}+(-1)^{n}4r^{\prime2}+1\right)  r^{\prime\prime
}\sinh w+(-1)^{n}r\left(  r^{\prime}\left(  (-1)^{n}+r^{\prime2}\right)
\left(  \sqrt{(-1)^{n}+r^{\prime2}}k_{2}^{\prime}+(-1)^{n}r^{\prime
\prime\prime}\sinh w\right)  +r^{\prime\prime2}\sinh w\right)  }{\left(
(-1)^{n}+r^{\prime2}\right)  ^{3/2}}}%
\end{array}
\right)  B_{3}\nonumber\\
&  +\left(
\begin{array}
[c]{l}%
{k_{3}\left(  r\left(  \sqrt{(-1)^{n}+r^{\prime2}}(k_{1}\cosh v\cosh
w-k_{2}\sinh w)+(-1)^{n}2r^{\prime\prime}\right)  +(-1)^{n}2r^{\prime
2}+1\right)  }\\
{+\frac{\left(
\begin{array}
[c]{l}%
\left(  3r^{\prime4}+(-1)^{n}4r^{\prime2}+1\right)  r^{\prime\prime}\sinh
v\cosh w\\
+(-1)^{n}r\left(
\begin{array}
[c]{l}%
k_{3}^{\prime}r^{\prime}\left(  (-1)^{n}+r^{\prime2}\right)  ^{3/2}\\
+\left(  r^{\prime\prime2}+(-1)^{n}r^{\prime}r^{\prime\prime\prime}\left(
(-1)^{n}+r^{\prime2}\right)  \right)  \sinh v\cosh w
\end{array}
\right)
\end{array}
\right)  }{\left(  (-1)^{n}+r^{\prime2}\right)  ^{3/2}}}\\
{-}r{\sqrt{(-1)^{n}+r^{\prime2}}k_{3}^{2}\sinh v\cosh w}%
\end{array}
\right)  B_{4},\nonumber
\end{align}%
\begin{align}
(C_{n})_{uv}  &  =r{\sqrt{(-1)^{n}+r^{\prime2}}(k_{1}\sinh v-k_{3}\cosh
v)\cosh w}B_{1}\label{C12uv}\\
&  +{\frac{r^{\prime}\left(  (-1)^{n}+r^{\prime2}+rr^{\prime\prime}\right)
}{\sqrt{(-1)^{n}+r^{\prime2}}}}\left(  \sinh v\cosh wB_{2}+\cosh v\cosh
wB_{4}\right)  ,\nonumber
\end{align}%
\begin{equation}
(C_{n})_{vw}=r\sqrt{(-1)^{n}+r^{\prime2}}\left(  {\sinh v\sinh w}B_{2}+{\cosh
v\sinh w}B_{4}\right)  , \label{C12vw}%
\end{equation}%
\begin{equation}
(C_{n})_{vv}=r\sqrt{(-1)^{n}+r^{\prime2}}\left(  {\cosh v\cosh w}B_{2}+{\sinh
v\cosh w}B_{4}\right)  , \label{C12vv}%
\end{equation}%
\begin{align}
(C_{n})_{uw}  &  ={-}r{\sqrt{(-1)^{n}+r^{\prime2}}((k_{3}\sinh v-k_{1}\cosh
v)\sinh w+k_{2}\cosh w)}B_{1}\label{C12uw}\\
&  +{\frac{r^{\prime}\left(  (-1)^{n}+r^{\prime2}+rr^{\prime\prime}\right)
}{\sqrt{(-1)^{n}+r^{\prime2}}}}\left(  \cosh v\sinh wB_{2}+\cosh wB_{3}+\sinh
v\sinh wB_{4}\right) \nonumber
\end{align}
and%
\begin{equation}
(C_{n})_{ww}=r\sqrt{(-1)^{n}+r^{\prime2}}\left(  {\cosh v\cosh w}B_{2}+{\sinh
w}B_{3}+{\sinh v\cosh w}B_{4}\right)  , \label{C12ww}%
\end{equation}
where $\mathcal{A}_{u}=\frac{\partial\mathcal{A}(u,v,w)}{\partial u}$.

The coefficients of the second fundamental forms and their determinants for
$C_{1}$ and $C_{2}$ can be obtained with the aid of (\ref{hij}), (\ref{NC12})
and (\ref{C12uu})-(\ref{C12ww}) as%
\begin{equation}
\left.
\begin{array}
[c]{l}%
\left(  h_{11}\right)  _{n}=\frac{\left(
\begin{array}
[c]{l}%
\left(  (-1)^{n}+r^{\prime2}\right)  \left(
\begin{array}
[c]{l}%
r\mathcal{A}^{2}+r^{\prime\prime}+(-1)^{n}\sqrt{(-1)^{n}+r^{\prime2}%
}\mathcal{A}\\
+\frac{(-1)^{n}rr^{\prime2}}{4}\left(
\begin{array}
[c]{l}%
-8k_{1}(k_{2}\sinh w+k_{3}\sinh v\cosh w)\cosh v\cosh w\\
+4k_{1}^{2}\left(  \cosh^{2}v\cosh^{2}w-1\right) \\
+4k_{2}(k_{2}\cosh w+2k_{3}\sinh v\sinh w)\cosh w\\
+k_{3}^{2}\left(  2\sinh^{2}v\cosh(2w)+\cosh(2v)+3\right)
\end{array}
\right)
\end{array}
\right) \\
+(-1)^{n}2rr^{\prime\prime}\sqrt{(-1)^{n}+r^{\prime2}}\mathcal{A}%
+rr^{\prime\prime2}%
\end{array}
\right)  }{(-1)^{n}+r^{\prime2}},\\
\\
\left(  h_{12}\right)  _{n}=\left(  h_{21}\right)  _{n}=rr^{\prime}%
\sqrt{(-1)^{n}+r^{\prime2}}(k_{3}\cosh v-k_{1}\sinh v)\cosh w,\\
\\
\left(  h_{13}\right)  _{n}=\left(  h_{31}\right)  _{n}=rr^{\prime}%
\sqrt{(-1)^{n}+r^{\prime2}}((k_{3}\sinh v-k_{1}\cosh v)\sinh w+k_{2}\cosh
w),\\
\\
\left(  h_{22}\right)  _{n}=(-1)^{n}r\left(  (-1)^{n}+r^{\prime2}\right)
\cosh^{2}w,\text{ \ }\left(  h_{23}\right)  _{n}=\left(  h_{32}\right)
_{n}=0,\text{ \ }\left(  h_{33}\right)  _{n}=(-1)^{n}r\left(  (-1)^{n}%
+r^{\prime2}\right)
\end{array}
\right\}  \label{hijC12}%
\end{equation}
and%
\begin{equation}
\det[h_{ij}]_{n}=r^{2}\left(  (-1)^{n}+r^{\prime2}\right)  \left(
\begin{array}
[c]{l}%
(-1)^{n}2rr^{\prime\prime}\sqrt{(-1)^{n}+r^{\prime2}}\mathcal{A}%
+rr^{\prime\prime2}\\
+\left(  (-1)^{n}+r^{\prime2}\right)  \left(  r^{\prime\prime}+\mathcal{A}%
\left(  r\mathcal{A}+(-1)^{n}\sqrt{(-1)^{n}+r^{\prime2}}\right)  \right)
\end{array}
\right)  \cosh^{2}w, \label{dethijC12}%
\end{equation}
respectively.

Thus, using (\ref{NC12}), (\ref{detgijC12}) and (\ref{dethijC12}) in
(\ref{Gaussformula}), we obtain the Gaussian curvatures $K_{1}$ and $K_{2}$ of
$C_{1}$ and $C_{2},$ respectively, as follows:

\begin{theorem}
The Gaussian curvatures of the canal hypersurfaces generated by a spacelike
curve with parallel timelike normal vector field $B_{2}$ parametrized by
(\ref{C12}) in $E_{1}^{4}$ are
\begin{equation}
K_{n}=-\frac{\left(  (-1)^{n}+r^{\prime2}\right)  \left(  r^{\prime\prime
}+(-1)^{n}\mathcal{A}\left(  (-1)^{n}r\mathcal{A}+\sqrt{(-1)^{n}+r^{\prime2}%
}\right)  \right)  +rr^{\prime\prime}\left(  (-1)^{n}2\sqrt{(-1)^{n}%
+r^{\prime2}}\mathcal{A}+r^{\prime\prime}\right)  }{r^{2}\left(  r\left(
(-1)^{n}\sqrt{(-1)^{n}+r^{\prime2}}\mathcal{A}+r^{\prime\prime}\right)
+(-1)^{n}+r^{\prime2}\right)  ^{2}}. \label{KC12}%
\end{equation}
So, we can prove the following theorem:
\end{theorem}

\begin{theorem}
\label{gaussteo}The canal hypersurfaces generated by a spacelike curve with
parallel timelike normal vector field $B_{2}$ parametrized by (\ref{C12}) in
$E_{1}^{4}$ are flat if and only if the curve $\gamma(u)$ which generates the
canal hypersurfaces is a straight line and the radius function $r(u)$
satisfies $r(u)=c_{1}u+c_{2}$ or $r(u)=\pm\sqrt{(-1)^{n}\left(  e^{2c_{1}%
}-\left(  u+c_{2}\right)  ^{2}\right)  },$ where $c_{1},c_{2}\in%
\mathbb{R}
$.
\end{theorem}

\begin{proof}
If the canal hypersurfaces $C_{n}$ parametrized by (\ref{C12}) are flat, then
from (\ref{KC12}) we get%
\begin{align}
&  \cosh^{2}v\cosh^{2}w\left(  k_{1}^{2}r((-1)^{n}+r^{\prime2})\right)
+\sinh^{2}w\left(  k_{2}^{2}r((-1)^{n}+r^{\prime2})\right)  +\sinh^{2}%
v\cosh^{2}w\left(  k_{3}^{2}r((-1)^{n}+r^{\prime2})\right) \nonumber\\
&  +\cosh v\sinh w\cosh w\left(  -2k_{1}k_{2}r((-1)^{n}+r^{\prime2})\right)
+\sinh v\cosh v\cosh^{2}w\left(  -2k_{1}k_{3}r((-1)^{n}+r^{\prime2})\right)
\nonumber\\
&  +\sinh v\sinh w\cosh w\left(  2k_{2}k_{3}r((-1)^{n}+r^{\prime2})\right)
\nonumber\\
&  +\cosh v\cosh w\left(  k_{1}\sqrt{(-1)^{n}+r^{\prime2}}(1+(-1)^{n}%
r^{\prime2}+(-1)^{n}2rr^{\prime\prime})\right) \nonumber\\
&  +\sinh w\left(  -k_{2}\sqrt{(-1)^{n}+r^{\prime2}}(1+(-1)^{n}r^{\prime
2}+(-1)^{n}2rr^{\prime\prime})\right) \nonumber\\
&  +\sinh v\cosh w\left(  -k_{3}\sqrt{(-1)^{n}+r^{\prime2}}(1+(-1)^{n}%
r^{\prime2}+(-1)^{n}2rr^{\prime\prime})\right) \nonumber\\
&  +r^{\prime\prime}((-1)^{n}+r^{\prime2}+rr^{\prime\prime})=0.
\label{C12Gaussacik}%
\end{align}
Now, let us see that the set
\begin{align*}
L  &  =\{\cosh^{2}v\cosh^{2}w,\sinh^{2}w,\sinh^{2}v\cosh^{2}w,\\
&  \cosh v\sinh w\cosh w,\sinh v\cosh v\cosh^{2}w,\sinh v\sinh w\cosh w,\\
&  \cosh v\cosh w,\sinh w,\sinh v\cosh w,1\}
\end{align*}
is linear independent. Let us suppose that%
\begin{align}
&  a\cosh^{2}v\cosh^{2}w+b\sinh^{2}w+c\sinh^{2}v\cosh^{2}w\nonumber\\
&  +d\cosh v\sinh w\cosh w+e\sinh v\cosh v\cosh^{2}w+f\sinh v\sinh w\cosh
w\nonumber\\
&  +k\cosh v\cosh w+m\sinh w+n\sinh v\cosh w+p=0 \label{ind1}%
\end{align}
holds for constants $a,b,c,d,e,f,k,m,n$ and $p=0$. From (\ref{ind1}), we get%
\begin{align}
&  \cosh^{2}w\left(  a\cosh^{2}v+c\sinh^{2}v+e\sinh v\cosh v\right)  +\cosh
w\left(  k\cosh v+n\sinh v\right) \nonumber\\
\text{ }  &  +\sinh w\cosh w\left(  d\cosh v+f\sinh v\right)  +b\sinh
^{2}w+m\sinh w=0. \label{ind2}%
\end{align}
Since the set $\{\cosh^{2}w,\cosh w,\sinh w\cosh w,\sinh^{2}w,\sinh w\}$ is
linear independent, we have%
\begin{equation}
\left.
\begin{array}
[c]{l}%
a\cosh^{2}v+c\sinh^{2}v+e\sinh v\cosh v=0,\\
k\cosh v+n\sinh v=0,\\
d\cosh v+f\sinh v=0,\\
b=m=0.
\end{array}
\right\}  \label{ind3}%
\end{equation}
From (\ref{ind3}), we get $a=c=e=k=n=d=f=b=m=0$ and thus, we see that the set
$L$ is linear independent. So, we get%
\begin{equation}
\left.
\begin{array}
[c]{l}%
k_{1}^{2}r((-1)^{n}+r^{\prime2})=0,\text{ }k_{2}^{2}r((-1)^{n}+r^{\prime
2})=0,\text{ }k_{3}^{2}r((-1)^{n}+r^{\prime2})=0,\\
\\
-2k_{1}k_{2}r((-1)^{n}+r^{\prime2})=0,\text{ }-2k_{1}k_{3}r((-1)^{n}%
+r^{\prime2})=0,\text{ }2k_{2}k_{3}r((-1)^{n}+r^{\prime2})=0,\\
\\
k_{1}\sqrt{(-1)^{n}+r^{\prime2}}(1+(-1)^{n}r^{\prime2}+(-1)^{n}2rr^{\prime
\prime})=0,\text{ }\\
\\
-k_{2}\sqrt{(-1)^{n}+r^{\prime2}}(1+(-1)^{n}r^{\prime2}+(-1)^{n}%
2rr^{\prime\prime})=0,\\
\\
-k_{3}\sqrt{(-1)^{n}+r^{\prime2}}(1+(-1)^{n}r^{\prime2}+(-1)^{n}%
2rr^{\prime\prime})=0,\text{ }\\
\\
r^{\prime\prime}((-1)^{n}+r^{\prime2}+rr^{\prime\prime})=0.
\end{array}
\right\}  \label{ind4}%
\end{equation}
Thus from (\ref{ind4}), we have $k_{1}=k_{2}=k_{3}=0$, it implies from
(\ref{kappa2}) that the curvature of the curve $\gamma(u)$ is identically zero
(\cite{Melek}). The last equation of (\ref{ind4}) leads to%
\[
r(u)=c_{1}u+c_{2}%
\]
or%
\[
r(u)=\pm\sqrt{(-1)^{n}\left(  e^{2c_{1}}-\left(  u+c_{2}\right)  ^{2}\right)
}.
\]
So, the proof completes.
\end{proof}

If we obtain the inverse matrices of the first fundamental forms of
(\ref{C12}) from (\ref{gijC12}) and use these with (\ref{hijC12}) in
(\ref{sekilop}), then we get the components of the shape operators $S_{n}$ of
(\ref{C12}) as%
\begin{equation}
\left.
\begin{array}
[c]{l}%
\left(  S_{11}\right)  _{n}=-\frac{\left(
\begin{array}
[c]{l}%
(-1)^{n}k_{1}\left(
\begin{array}
[c]{l}%
2r\left(  (-1)^{n}+r^{\prime2}\right)  \left(  k_{2}\sinh w+k_{3}\sinh v\cosh
w\right) \\
-(-1)^{n}\sqrt{(-1)^{n}+r^{\prime2}}\left(  (-1)^{n}+r^{\prime2}%
+2rr^{\prime\prime}\right)
\end{array}
\right)  \cosh v\cosh w\\
-(-1)^{n}r\left(  (-1)^{n}+r^{\prime2}\right)  k_{1}^{2}\cosh^{2}v\cosh^{2}w\\
+k_{3}\left(
\begin{array}
[c]{l}%
-(-1)^{n}2r\left(  (-1)^{n}+r^{\prime2}\right)  k_{2}\sinh w\\
+\sqrt{(-1)^{n}+r^{\prime2}}\left(  (-1)^{n}+r^{\prime2}+2rr^{\prime\prime
}\right)
\end{array}
\right)  \sinh v\cosh w\\
-(-1)^{n}r\left(  (-1)^{n}+r^{\prime2}\right)  \left(  k_{2}^{2}\sinh
^{2}w+k_{3}^{2}\sinh^{2}v\cosh^{2}w\right) \\
+\sqrt{(-1)^{n}+r^{\prime2}}\left(  (-1)^{n}+r^{\prime2}+2rr^{\prime\prime
}\right)  k_{2}\sinh w-(-1)^{n}r^{\prime\prime}\left(  (-1)^{n}+r^{\prime
2}+rr^{\prime\prime}\right)
\end{array}
\right)  }{\left(  r\left(  (-1)^{n}\sqrt{(-1)^{n}+r^{\prime2}}\mathcal{A}%
+r^{\prime\prime}\right)  +(-1)^{n}+r^{\prime2}\right)  ^{2}},\\
\\
\left(  S_{21}\right)  _{n}=\frac{r^{\prime}\sqrt{(-1)^{n}+r^{\prime2}}%
(k_{3}\cosh v-k_{1}\sinh v)\text{sech}w}{r\left(  r\left(  (-1)^{n}%
\sqrt{(-1)^{n}+r^{\prime2}}\mathcal{A}+r^{\prime\prime}\right)  +(-1)^{n}%
+r^{\prime2}\right)  },\text{ \ }\left(  S_{31}\right)  _{n}=\frac{r^{\prime
}\sqrt{(-1)^{n}+r^{\prime2}}((k_{3}\sinh v-k_{1}\cosh v)\sinh w+k_{2}\cosh
w)}{r\left(  r\left(  (-1)^{n}\sqrt{(-1)^{n}+r^{\prime2}}\mathcal{A}%
+r^{\prime\prime}\right)  +(-1)^{n}+r^{\prime2}\right)  },\\
\\
\left(  S_{22}\right)  _{n}=\left(  S_{33}\right)  _{n}=\frac{(-1)^{n}}%
{r},\text{ \ }\left(  S_{12}\right)  _{n}=\left(  S_{13}\right)  _{n}=\left(
S_{23}\right)  _{n}=\left(  S_{32}\right)  _{n}=0.
\end{array}
\right\}  \label{SijC12}%
\end{equation}
So, from (\ref{meanformula}), (\ref{NC12}) and (\ref{SijC12}), we have

\begin{theorem}
The mean curvatures of the canal hypersurfaces generated by a spacelike curve
with parallel timelike normal vector field $B_{2}$ parametrized by (\ref{C12})
in $E_{1}^{4}$ are
\begin{equation}
H_{n}=\frac{\left(
\begin{array}
[c]{l}%
r\left(
\begin{array}
[c]{l}%
r\left(  (-1)^{n}+r^{\prime2}\right)  \left(  -(-1)^{n}\left(  \left(
k_{1}^{2}\cosh^{2}v+k_{3}^{2}\sinh^{2}v\right)  \cosh^{2}w+k_{2}^{2}\sinh
^{2}w\right)  \right) \\
+(-1)^{n}\left(
\begin{array}
[c]{l}%
2r\left(  (-1)^{n}+r^{\prime2}\right)  \left(  k_{2}\sinh w+k_{3}\sinh v\cosh
w\right) \\
-(-1)^{n}\sqrt{(-1)^{n}+r^{\prime2}}\left(  2rr^{\prime\prime}+r^{\prime
2}+(-1)^{n}\right)
\end{array}
\right)  k_{1}\cosh v\cosh w\\
+\left(  \sqrt{(-1)^{n}+r^{\prime2}}\left(  2rr^{\prime\prime}+r^{\prime
2}+(-1)^{n}\right)  -(-1)^{n}2r\left(  (-1)^{n}+r^{\prime2}\right)  k_{2}\sinh
w\right)  k_{3}\sinh v\cosh w\\
+\sqrt{(-1)^{n}+r^{\prime2}}\left(  2rr^{\prime\prime}+r^{\prime2}%
+(-1)^{n}\right)  k_{2}\sinh w-(-1)^{n}r^{\prime\prime}\left(  rr^{\prime
\prime}+r^{\prime2}+(-1)^{n}\right)
\end{array}
\right) \\
-(-1)^{n}2\left(  r\left(  (-1)^{n}\sqrt{(-1)^{n}+r^{\prime2}}\mathcal{A}%
+r^{\prime\prime}\right)  +r^{\prime2}+(-1)^{n}\right)  ^{2}%
\end{array}
\right)  }{(-1)^{n}3r\left(  r\left(  (-1)^{n}\sqrt{(-1)^{n}+r^{\prime2}%
}\mathcal{A}+r^{\prime\prime}\right)  +r^{\prime2}+(-1)^{n}\right)  ^{2}}.
\label{HC12}%
\end{equation}
Thus, we get
\end{theorem}

\begin{theorem}
The canal hypersurfaces generated by a spacelike curve with parallel timelike
normal vector field $B_{2}$ parametrized by (\ref{C12}) in $E_{1}^{4}$ are
minimal if and only if the curve $\gamma(u)$ which generates the canal
hypersurfaces is a straight line and the radius function $r(u)$ satisfies
$r(u)=\pm\sqrt{-(-1)^{n}\left(  -e^{2c_{1}}+\left(  u+c_{2}\right)
^{2}\right)  }$ or $r\ $satisfies $\int\frac{dr}{\sqrt{\left(  \frac{c_{3}}%
{r}\right)  ^{\frac{4}{3}}-(-1)^{n}}}=\pm u+c_{4}$, where $c_{1},c_{2}%
,c_{3},c_{4}\in%
\mathbb{R}
$.
\end{theorem}

\begin{proof}
If the canal hypersurfaces $C_{n}$ parametrized by (\ref{C12}) are minimal,
then from (\ref{HC12}) we get%
\begin{align}
&  \cosh^{2}v\cosh^{2}w\left(  -(-1)^{n}3k_{1}^{2}r^{2}((-1)^{n}+r^{\prime
2})\right)  +\sinh^{2}w\left(  -(-1)^{n}3k_{2}^{2}r^{2}((-1)^{n}+r^{\prime
2})\right) \nonumber\\
&  +\sinh^{2}v\cosh^{2}w\left(  -(-1)^{n}3k_{3}^{2}r^{2}((-1)^{n}+r^{\prime
2})\right) \nonumber\\
&  +\cosh v\sinh w\cosh w\left(  (-1)^{n}6k_{1}k_{2}r^{2}((-1)^{n}+r^{\prime
2})\right) \nonumber\\
&  +\sinh v\cosh v\cosh^{2}w\left(  (-1)^{n}6k_{1}k_{3}r^{2}((-1)^{n}%
+r^{\prime2})\right) \nonumber\\
&  +\sinh v\sinh w\cosh w\left(  -(-1)^{n}6k_{2}k_{3}r^{2}((-1)^{n}%
+r^{\prime2})\right) \nonumber\\
&  +\cosh v\cosh w\left(  -(-1)^{n}k_{1}r\sqrt{(-1)^{n}+r^{\prime2}%
}(5+(-1)^{n}5r^{\prime2}+(-1)^{n}6rr^{\prime\prime})\right) \nonumber\\
&  +\sinh w\left(  (-1)^{n}k_{2}r\sqrt{(-1)^{n}+r^{\prime2}}(5+(-1)^{n}%
5r^{\prime2}+(-1)^{n}6rr^{\prime\prime})\right) \nonumber\\
&  +\sinh v\cosh w\left(  (-1)^{n}k_{3}r\sqrt{(-1)^{n}+r^{\prime2}}%
(5+(-1)^{n}5r^{\prime2}+(-1)^{n}6rr^{\prime\prime})\right) \nonumber\\
&  -(-1)^{n}((-1)^{n}+r^{\prime2}+rr^{\prime\prime})(2(-1)^{n}+2r^{\prime
2}+3rr^{\prime\prime})=0. \label{C12Meanacik}%
\end{align}
With similar procedure in the proof of Theorem \ref{gaussteo}, we get%
\begin{equation}
\left.
\begin{array}
[c]{l}%
-(-1)^{n}3k_{1}^{2}r^{2}((-1)^{n}+r^{\prime2})=0,\text{ }-(-1)^{n}3k_{2}%
^{2}r^{2}((-1)^{n}+r^{\prime2})=0,\text{ }-(-1)^{n}3k_{3}^{2}r^{2}%
((-1)^{n}+r^{\prime2})=0,\\
\\
(-1)^{n}6k_{1}k_{2}r^{2}((-1)^{n}+r^{\prime2})=0,\text{ }(-1)^{n}6k_{1}%
k_{3}r^{2}((-1)^{n}+r^{\prime2})=0,\text{ }-(-1)^{n}6k_{2}k_{3}r^{2}%
((-1)^{n}+r^{\prime2})=0,\\
\\
-(-1)^{n}k_{1}r\sqrt{(-1)^{n}+r^{\prime2}}(5+(-1)^{n}5r^{\prime2}%
+(-1)^{n}6rr^{\prime\prime})=0,\\
\\
(-1)^{n}k_{2}r\sqrt{(-1)^{n}+r^{\prime2}}(5+(-1)^{n}5r^{\prime2}%
+(-1)^{n}6rr^{\prime\prime})=0,\\
\\
(-1)^{n}k_{3}r\sqrt{(-1)^{n}+r^{\prime2}}(5+(-1)^{n}5r^{\prime2}%
+(-1)^{n}6rr^{\prime\prime})=0,\\
\\
-(-1)^{n}((-1)^{n}+r^{\prime2}+rr^{\prime\prime})(2(-1)^{n}+2r^{\prime
2}+3rr^{\prime\prime})=0.
\end{array}
\right\}  \label{ind5}%
\end{equation}
Thus from (\ref{ind5}), we have $k_{1}=k_{2}=k_{3}=0$ that is, the curvature
of the curve $\gamma(u)$ vanishes from (\ref{kappa2}) (\cite{Melek}). Also,
the last equation of (\ref{ind5}) implies%
\[
r(u)=\pm\sqrt{-(-1)^{n}\left(  -e^{2c_{1}}+\left(  u+c_{2}\right)
^{2}\right)  }%
\]
or $r$ satisfies the differential equation%
\[
2(-1)^{n}+2r^{\prime2}+3rr^{\prime\prime}=0
\]
(for solution of this differential equation, one can see \cite{Kazan}) and the
proof completes.
\end{proof}

From (\ref{KC12}) and (\ref{HC12}), we reach the following important relation
between the Gaussian curvatures and mean curvatures of the canal hypersurfaces
(\ref{C12}):

\begin{theorem}
The Gaussian and mean curvatu\i res of the canal hypersurfaces $C_{n}$
parametrized by (\ref{C12}) in $E_{1}^{4}$ satisfy%
\[
3H_{n}-r^{2}K_{n}+\frac{2}{r}=0.
\]

\end{theorem}

Using (\ref{SijC12}) in the equation of $\det(S_{n}-cI)=0$ and solve this, we
obtain the principal curvatures of $C_{n}$ parametrized by (\ref{C12}) as follows:

\begin{theorem}
The principal curvatures of the canal hypersurfaces generated by a spacelike
curve with parallel timelike normal vector field $B_{2}$ parametrized by
(\ref{C12}) in $E_{1}^{4}$ are%
\begin{equation}
\left.
\begin{array}
[c]{l}%
\left(  c_{1}\right)  _{n}=\left(  c_{2}\right)  _{n}=\frac{(-1)^{n}}{r},\\
\\
\left(  c_{3}\right)  _{n}=\frac{\left(
\begin{array}
[c]{l}%
-(-1)^{n}\left(
\begin{array}
[c]{l}%
2r\left(  (-1)^{n}+r^{\prime2}\right)  \left(  k_{2}\sinh w+k_{3}\sinh v\cosh
w\right) \\
-(-1)^{n}\sqrt{(-1)^{n}+r^{\prime2}}\left(  2rr^{\prime\prime}+r^{\prime
2}+(-1)^{n}\right)
\end{array}
\right)  k_{1}\cosh v\cosh w\\
+(-1)^{n}r\left(  (-1)^{n}+r^{\prime2}\right)  k_{1}^{2}\cosh^{2}v\cosh^{2}w\\
-\left(  \sqrt{(-1)^{n}+r^{\prime2}}\left(  2rr^{\prime\prime}+r^{\prime
2}+(-1)^{n}\right)  -(-1)^{n}2r\left(  (-1)^{n}+r^{\prime2}\right)  k_{2}\sinh
w\right)  k_{3}\sinh v\cosh w\\
+(-1)^{n}r\left(  (-1)^{n}+r^{\prime2}\right)  \left(  k_{2}^{2}\sinh
^{2}w+k_{3}^{2}\sinh^{2}v\cosh^{2}w\right) \\
-\sqrt{(-1)^{n}+r^{\prime2}}\left(  2rr^{\prime\prime}+r^{\prime2}%
+(-1)^{n}\right)  k_{2}\sinh w+(-1)^{n}r^{\prime\prime}\left(  rr^{\prime
\prime}+r^{\prime2}+(-1)^{n}\right)
\end{array}
\right)  }{\left(  r\left(  (-1)^{n}\sqrt{(-1)^{n}+r^{\prime2}}\mathcal{A}%
+r^{\prime\prime}\right)  +r^{\prime2}+(-1)^{n}\right)  ^{2}}.
\end{array}
\right\}  \label{asliC12}%
\end{equation}

\end{theorem}

Moreover, a hypersurface is called $(H,K)_{uv}$-Weingarten or $(H,K)_{uw}%
$-Weingarten or $(H,K)_{vw}$-Weingarten, if the equation%
\begin{equation}
H_{u}K_{v}-H_{v}K_{u}=0\text{ or }H_{u}K_{w}-H_{w}K_{u}=0\text{ or }H_{v}%
K_{w}-H_{w}K_{v}=0 \label{w1}%
\end{equation}
holds on it, respectively. Here $H_{u}=\frac{\partial H}{\partial u},$
$H_{v}=\frac{\partial H}{\partial v},$ and so on. Thus, from (\ref{KC12}) and
(\ref{HC12}) we have

\begin{theorem}
The canal hypersurfaces generated by a spacelike curve with parallel timelike
normal vector field $B_{2}$ parametrized by (\ref{C12}) in $E_{1}^{4}$ are

i) $(H_{n},K_{n})_{vw}$-Weingarten;

ii) $(H_{n},K_{n})_{uv}$-Weingarten if and only if $k_{1}=k_{3}=0$;

iii) $(H_{n},K_{n})_{uw}$-Weingarten if and only if the curve $\gamma(u)$
which generates the canal hypersurfaces is a straight line.
\end{theorem}

\begin{proof}
From (\ref{KC12}) and (\ref{HC12}) we have%
\[
\left.
\begin{array}
[c]{l}%
(H_{n})_{v}(K_{n})_{w}-(H_{n})_{w}(K_{n})_{v}=0,\\
\\
(H_{n})_{u}(K_{n})_{v}-(H_{n})_{v}(K_{n})_{u}=(-1)^{n}\frac{2r^{\prime}\left(
(-1)^{n}+r^{\prime2}\right)  ^{5/2}(k_{3}\cosh v-k_{1}\sinh v)\cosh w}%
{3r^{4}\left(  r\left(  (-1)^{n}\sqrt{(-1)^{n}+r^{\prime2}}\mathcal{A}%
+r^{\prime\prime}\right)  +(-1)^{n}+r^{\prime2}\right)  ^{3}},\\
\\
(H_{n})_{u}(K_{n})_{w}-(H_{n})_{w}(K_{n})_{u}=(-1)^{n}\frac{2r^{\prime}\left(
(-1)^{n}+r^{\prime2}\right)  ^{5/2}((k_{3}\sinh v-k_{1}\cosh v)\sinh
w+k_{2}\cosh w)}{3r^{4}\left(  r\left(  (-1)^{n}\sqrt{(-1)^{n}+r^{\prime2}%
}\mathcal{A}+r^{\prime\prime}\right)  +(-1)^{n}+r^{\prime2}\right)  ^{3}}%
\end{array}
\right\}
\]
and these equations completes the proof.
\end{proof}

After the above detailed calculations and results given for the canal
hypersurfaces $C_{1}$ and $C_{2}$, now let us give the general expression of
the eight types of canal hypersurfaces generated by non-null curves with
parallel frame in $E_{1}^{4}$ and give some important geometric
characterizations for them. Here, we will include the canal hypersurfaces
$C_{1}$ and $C_{2},$ which we gave above, in the general expressions and the
following general results will also include the results related to these
hypersurfaces, too. The proofs of the following results can be done seperately
with similar procedure given for $C_{1}$ and $C_{2}$ above.

Here we must note that, from now on $C_{1}$ and $C_{2}$ state the canal
hypersurfaces generated by a spacelike curve with parallel timelike normal
vector field $B_{2}$; $C_{3}$ and $C_{4}$ state the canal hypersurfaces
generated by a spacelike curve with parallel timelike normal vector field
$B_{3}$; $C_{5}$ and $C_{6}$ state the canal hypersurfaces generated by a
spacelike curve with parallel timelike normal vector field $B_{4}$; $C_{7}$
and $C_{8}$ state the canal hypersurfaces generated by a timelike curve with
parallel frame. Also, $C_{1}$, $C_{3}$, $C_{5}$, $C_{7}$ and $C_{2}$, $C_{4}$,
$C_{6}$, $C_{8}$ state the canal hypersurfaces that are formed as the envelope
of a family of pseudo hyperspheres and pseudo hyperbolic hyperspheres,
respectively, whose centers lie on a non-null curve $\gamma(s)$ with parallel
vector fields $B_{i}$. On the other hand, the functions $f_{i}(u,v,w)$ and the
constants $\varepsilon_{i}$, $\mu_{i}$ ($i=1,2,3$) and $\eta$ will be as
follows:%
\begin{equation}%
\begin{tabular}
[c]{|c|c|c|c|}\hline
Hypersurfaces $C_{m}$ & $f_{1}(u,v,w)$ & $f_{2}(u,v,w)$ & $f_{3}%
(u,v,w)$\\\hline\hline
$C_{1}$ and $C_{2}$ & $\cosh v\cosh w$ & $\sinh w$ & $\sinh v\cosh w$\\\hline
$C_{3}$ and $C_{4}$ & $\sinh v\cosh w$ & $\cosh v\cosh w$ & $\sinh w$\\\hline
$C_{5}$ and $C_{6}$ & $\sinh w$ & $\sinh v\cosh w$ & $\cosh v\cosh w$\\\hline
$C_{7}$ and $C_{8}$ & $\cos v\cos w$ & $\sin v\cos w$ & $\sin w$\\\hline
\end{tabular}
\ \ \ \ \ \label{filer}%
\end{equation}%
\begin{equation}%
\begin{tabular}
[c]{|c|c|c|c|c|c|c|c|}\hline
Hypersurfaces $C_{m}$ & $\varepsilon_{1}$ & $\varepsilon_{2}$ & $\varepsilon
_{3}$ & $\mu_{1}$ & $\mu_{2}$ & $\mu_{3}$ & $\eta$\\\hline\hline
$C_{1}$ & $1$ & $-1$ & $-1$ & $-1$ & $1$ & $1$ & $-1$\\\hline
$C_{2}$ & $1$ & $-1$ & $-1$ & $1$ & $-1$ & $-1$ & $-1$\\\hline
$C_{3}$ & $1$ & $-1$ & $1$ & $1$ & $-1$ & $1$ & $-1$\\\hline
$C_{4}$ & $1$ & $-1$ & $1$ & $-1$ & $1$ & $-1$ & $-1$\\\hline
$C_{5}$ & $1$ & $1$ & $-1$ & $1$ & $1$ & $-1$ & $1$\\\hline
$C_{6}$ & $1$ & $1$ & $-1$ & $-1$ & $-1$ & $1$ & $-1$\\\hline
$C_{7}$ & $1$ & $1$ & $1$ & $1$ & $1$ & $1$ & $1$\\\hline
$C_{8}$ & $1$ & $1$ & $1$ & $-1$ & $-1$ & $-1$ & $-1$\\\hline
\end{tabular}
\ \ \ \ \label{sabitler}%
\end{equation}

\begin{theorem}
Canal hypersurfaces generated by non-null curves with parallel frame in
$E_{1}^{4}$ can be parametrized by%
\begin{equation}
C_{m}(u,v,w)=\gamma(u)+(-1)^{m+(8-m)!}r(u)r^{\prime}(u)B_{1}(u)+r(u)\sqrt
{(-1)^{m+(8-m)!}+r^{\prime}(u)^{2}}\left(  \sum\limits_{i=2}^{4}%
f_{i-1}(u,v,w)B_{i}(u)\right)  , \label{Cmgenel}%
\end{equation}
where $m\in\{1,2,...,8\}$ and the functions $f_{i}(u,v,w)$ are given by
(\ref{filer}).
\end{theorem}

\begin{theorem}
The Gauss maps (i.e. unit normal vector fields) of the canal hypersurfaces
$C_{m}$ generated by non-null curves with parallel frame parametrized by
(\ref{Cmgenel}) in $E_{1}^{4}$ are%
\begin{equation}
N_{m}(u,v,w)=\eta\left(  (-1)^{(8-m)!}r^{\prime}B_{1}(u)+(-1)^{m}%
\sqrt{(-1)^{m+(8-m)!}+r^{\prime}(u)^{2}}\left(  \sum\limits_{i=2}^{4}%
f_{i-1}(u,v,w)B_{i}(u)\right)  \right)  ,\nonumber
\end{equation}
where $m\in\{1,2,...,8\}$; the functions $f_{i}(u,v,w)$ are given by
(\ref{filer}) and the constant $\eta$ is given by (\ref{sabitler}).
\end{theorem}

\begin{theorem}
The Gaussian curvatures of the canal hypersurfaces $C_{m}$ generated by
non-null curves with parallel frame parametrized by (\ref{Cmgenel}) in
$E_{1}^{4}$ are%
\begin{equation}
K_{m}=\frac{\left(
\begin{array}
[c]{l}%
r^{\prime\prime}\left(  \left(  -1\right)  ^{m+(8-m)!}+r^{\prime2}%
+rr^{\prime\prime}\right)  +\left(  \sum\limits_{i=1}^{3}\sum\limits_{j=1}%
^{3}\left(  \varepsilon_{i}\varepsilon_{j}k_{i}k_{j}r\left(  \left(
-1\right)  ^{m+(8-m)!}+r^{\prime2}\right)  f_{i}f_{j}\right)  \right) \\
+\left(  \sum\limits_{i=1}^{3}\left(  \mu_{i}k_{i}\sqrt{\left(  -1\right)
^{m+(8-m)!}+r^{\prime2}}\left(  \left(  -1\right)  ^{m+(8-m)!}+r^{\prime
2}+2rr^{\prime\prime}\right)  f_{i}\right)  \right)
\end{array}
\right)  }{\eta r^{2}\left(  r\left(  r^{\prime\prime}+\sum\limits_{i=1}%
^{3}\left(  \mu_{i}k_{i}\sqrt{\left(  -1\right)  ^{m+(8-m)!}+r^{\prime2}}%
f_{i}\right)  \right)  +\left(  -1\right)  ^{m+(8-m)!}+r^{\prime2}\right)
^{2}}, \label{TopluGauss}%
\end{equation}
where $m\in\{1,2,...,8\}$; the functions $f_{i}(u,v,w)$ are given by
(\ref{filer}) and the constants $\varepsilon_{i}$, $\mu_{i}$ ($i=1,2,3$) and
$\eta$ are given by (\ref{sabitler}).
\end{theorem}

\begin{theorem}
The canal hypersurfaces $C_{m}$ generated by non-null curves with parallel
frame parametrized by (\ref{Cmgenel}) in $E_{1}^{4}$ are flat if and only if
the curve $\gamma(u)$ which generates the canal hypersurfaces is a straight
line and the radius function $r(u)$ satisfies $r(u)=c_{1}u+c_{2}$ or
$r(u)=\pm\sqrt{(-1)^{m+(8-m)!}\left(  e^{2c_{1}}-\left(  u+c_{2}\right)
^{2}\right)  },$ where $c_{1},c_{2}\in%
\mathbb{R}
$.
\end{theorem}

\begin{theorem}
The mean curvatures of the canal hypersurfaces $C_{m}$ generated by non-null
curves with parallel frame parametrized by (\ref{Cmgenel}) in $E_{1}^{4}$ are%
\begin{equation}
H_{m}=\frac{\left(
\begin{array}
[c]{l}%
\left(  \left(  -1\right)  ^{m+(8-m)!}+r^{\prime2}+rr^{\prime\prime}\right)
\left(  2\left(  -1\right)  ^{m+(8-m)!}+2r^{\prime2}+3rr^{\prime\prime
}\right)  \\
+\left(  \sum\limits_{i=1}^{3}\sum\limits_{j=1}^{3}\left(  3\varepsilon
_{i}\varepsilon_{j}k_{i}k_{j}r^{2}\left(  \left(  -1\right)  ^{m+(8-m)!}%
+r^{\prime2}\right)  f_{i}f_{j}\right)  \right)  \\
+\left(  \sum\limits_{i=1}^{3}\left(  \mu_{i}k_{i}r\sqrt{\left(  -1\right)
^{m+(8-m)!}+r^{\prime2}}\left(  5\left(  \left(  -1\right)  ^{m+(8-m)!}%
+r^{\prime2}\right)  +6rr^{\prime\prime}\right)  f_{i}\right)  \right)
\end{array}
\right)  }{3\eta r\left(  r\left(  r^{\prime\prime}+\sum\limits_{i=1}%
^{3}\left(  \mu_{i}k_{i}\sqrt{\left(  -1\right)  ^{m+(8-m)!}+r^{\prime2}}%
f_{i}\right)  \right)  +\left(  -1\right)  ^{m+(8-m)!}+r^{\prime2}\right)
^{2}},\label{TopluMean}%
\end{equation}
where $m\in\{1,2,...,8\}$; the functions $f_{i}(u,v,w)$ are given by
(\ref{filer}) and the constants $\varepsilon_{i}$, $\mu_{i}$ ($i=1,2,3$) and
$\eta$ are given by (\ref{sabitler}).
\end{theorem}

\begin{theorem}
The canal hypersurfaces $C_{m}$ generated by non-null curves with parallel
frame parametrized by (\ref{Cmgenel}) in $E_{1}^{4}$ are minimal if and only
if the curve $\gamma(u)$ which generates the canal hypersurfaces is a straight
line and the radius function $r(u)$ satisfies $r(u)=\pm\sqrt{(-1)^{m+(8-m)!}%
\left(  e^{2c_{1}}-\left(  u+c_{2}\right)  ^{2}\right)  }$ or $r\ $satisfies
$\int\frac{dr}{\sqrt{\left(  \frac{c_{3}}{r}\right)  ^{\frac{4}{3}}-\left(
-1\right)  ^{m+(8-m)!}}}=\pm u+c_{4}$, where $c_{1},c_{2},c_{3},c_{4}\in%
\mathbb{R}
$.
\end{theorem}

\begin{theorem}
The Gaussian and mean curvatu\i res of the canal hypersurfaces $C_{m}$
generated by non-null curves with parallel frame parametrized by
(\ref{Cmgenel}) in $E_{1}^{4}$ satisfy%
\[
3H_{m}-r^{2}K_{m}-\frac{2\eta}{r}=0,
\]
where $m\in\{1,2,...,8\}$ and the constant $\eta$ is given by (\ref{sabitler}).
\end{theorem}

\begin{theorem}
The principal curvatures of the canal hypersurfaces $C_{m}$ generated by
non-null curves with parallel frame parametrized by (\ref{Cmgenel}) in
$E_{1}^{4}$ are%
\begin{equation}
\left(  c_{1}\right)  _{m}=\left(  c_{2}\right)  _{m}=\frac{(-1)^{m+1}\eta}%
{r},\text{ }\left(  c_{3}\right)  _{m}=(-1)^{m+1}r^{2}K_{m},
\label{AsliegGenel}%
\end{equation}
where $m\in\{1,2,...,8\}$ and $K_{m}$ are the Gaussian curvatures given by
(\ref{TopluGauss}).
\end{theorem}

\begin{theorem}
The canal hypersurfaces $C_{m}$ generated by non-null curves with parallel
frame parametrized by (\ref{Cmgenel}) in $E_{1}^{4}$ are

i) $(H_{m},K_{m})_{vw}$-Weingarten;

ii) $(H_{m},K_{m})_{uv}$-Weingarten if and only if $k_{1}=k_{3}=0$;

iii) $(H_{m},K_{m})_{uw}$-Weingarten if and only if the curve $\gamma(u)$
which generates the canal hypersurfaces is a straight line.
\end{theorem}

\end{document}